\documentclass[preprint,12pt]{elsarticle}

\usepackage{amssymb,amsfonts,amsmath,amsthm}
\usepackage{graphicx}
\usepackage{natbib}
\usepackage[dvipsnames]{xcolor}
\usepackage{float}

\newtheorem{theorem}{Theorem}[section]
\newtheorem{proposition}[theorem]{Proposition}
\newtheorem{lemma}[theorem]{Lemma}

\newtheorem{definition}[theorem]{Definition}

\journal{.}

\begin{document}

\begin{frontmatter}

\title{On affine homogeneous polynomial centrosymmetric matrices} 

 \author[manoel]{Miriam Manoel}
  \author[nery]{Leandro Nery}
  \affiliation[manoel]{organization={Department of Mathematics, ICMC,
  		University of São Paulo},
             addressline={13560-970},
             city={São Carlos},
             postcode={Caixa Postal 668},
            state={São Paulo},
             country={Brazil}}
 \affiliation[nery]{organization={Department of Mathematics, CCET,
 		Federal University of São Carlos},
             addressline={13565-905},
             city={São Carlos},
             postcode={Caixa Postal 676},
             state={São Paulo},
             country={Brazil}}

\begin{abstract}
		We explore a class of centrosymmetric matrices whose entries are polynomials in two variables, referred to as DNA matrices. Our motivation stems from an unexpected connection between these matrices and invariant polynomials under the action of a Lorentz rotation on the plane. Among several noteworthy properties, we establish that within a subclass of DNA matrices, singular matrices occur precisely when their order is even, and we determine their null space in such cases.  The results provide new insights into the structural characteristics of centrosymmetric matrices, expanding their theoretical foundations and potential applications. 
\end{abstract}


\end{frontmatter}

\section{Introduction} \label{sec: introduction}

The study of centrosymmetric matrices dates back to the 19th century, introduced by Zehfuss  (\cite{zehfuss1862zwei}). Such matrices, characterized by their invariance under central reflection, have played an important role in various areas of mathematics and applied sciences. Over the years, numerous researchers have explored their properties and applications, leading to significant theoretical developments and practical implementations.
A substantial body of research has explored the spectral properties, inverse structure, and computational aspects of centrosymmetric matrices (\cite{ANDREW1973151,cantoni1976eigenvalues, centro5, centro3, centro1, centro2, centro4, centro6}). Efficient algorithms have been developed to solve centrosymmetric linear systems (\cite{centro5, centro4}), and generalizations of centrosymmetric matrices have been treated by \cite{centro2}. Beyond theoretical considerations, centrosymmetric matrices have also found widespread applications; among many others, we cite \cite{centro7engenharia} for applications in engineering, \cite{centro8estatistica} in the statistical classification of data analysis, and  \cite{centro9datta} for an application in pattern recognition, with a connection between matrix symmetries and computational techniques for signal processing. More recently, these matrices have found applications in various fields, including algebraic curves, optimization, and biology (\cite{CS-Biologia,CS-Curvas24,BRACCIALI2024212,CS-Optimal25}).
\\

Recall that a square matrix $(a_{i,j})$ of order $n$ is centrosymmetric if
	\begin{eqnarray}
		a_{i,j} = a_{n+1-i, n+1-j}, \quad i,j = 1, \dots, n.
	\end{eqnarray}
This paper introduces a novel class of centrosymmetric matrices, the DNA matrices. As we shall see, these naturally appear in invariant theory, from manipulating real homogeneous polynomials defined on the Minkowski plane ${\mathbb R}^{1+1}$,
\[ f: {\mathbb R}^{1+1} \to {\mathbb R}, \]
that are invariant under the standard linear action of a hyperbolic rotation matrix $R_\theta \in {\bf O}(1,1)$, that is, 
\begin{equation} \label{eq: invariance}
 f(R_\theta x) = f(x), \quad x \in {\mathbb R}^{1+1}.
 \end{equation}
For details on invariant theory for linear action of Lorenz groups on  Minkowski spaces we refer to \cite{leandro}. We use the graded algebra structure of the ring of invariant polynomials and apply (\ref{eq: invariance}) on a general homogeneous polynomial of an arbitrary degree $n$. This ring is generated by 
$||x||_1^2$ (see Subsection~\ref{subseq: generator}).  The calculations result in a linear system of order $n+1$ whose  coefficient matrix  depends polynomially on two parameters $\alpha$ and $\beta$ that we now define. For practical purposes, the definition is presented for matrices of order $n+1$:

\begin{definition} \label{def: cs matrix}
For arbitrary  $\alpha$ and $\beta$, a square matrix $(a_{i,j})$ of order $n+1$ is a DNA matrix if its entries are given by the generating function
\footnotesize
\begin{eqnarray} \label{dna}
	a_{i,j} = \sum _{s=0}^{j-1} \binom{j-1}{s} \binom{n-j+1}{i-j+s} \alpha^{n-(i-j+2s)} \beta^{i-j+2s} + \sum _{r=0}^{j-1} \left( -1 \right) ^{r+1} \binom{j-r}{i} \binom{j}{r},
\end{eqnarray}   \normalsize
where $q$ assumes all integer values under the convention $\binom{p}{q} = 0$ whenever $p < q$ or $q < 0$.
  \end{definition}
\noindent This matrix has affine homogeneous polynomial entries: all entries of this matrix are homogeneous polynomials of degree $n$ except that on the diagonal we subtract 1. \\

We prove that a DNA matrix is centrosymmetric (Theorem~\ref{thm: DNA is CSM}), which is our main result. The nomenclature has been inspired by its shape resemblance to a DNA helix; see in Section~\ref{sec: sample computations} the explicit construction of the matrices of orders 2 to 7. As we see in Section~{\ref{sec: main}, the close relation with the rotational invariant homogeneous polynomials allows us to establish that the classification of a class of DNA matrices as singular or non-singular is determined precisely by the parity of their order: it is non-singular if, and only if, its order is odd. Furthermore, by explicitly determining the solution space of the associated linear system in
the singular case, we will also show that in the singular case the null space  has minimal dimension, namely 1.  \\

We emphasize that the centrosymmetry property that we encounter in the present paper holds for any DNA matrix with polynomial entries on two independent variables $\alpha$ and $\beta$, although these numbers were originally placed in context  as $\alpha = \cosh \theta$ and $\beta = \sinh \theta$.  In fact, this is an algebraic property inherited by the invariance condition (\ref{eq: invariance}) for a fixed $\theta$ and does not depend on the condition $\alpha^2 - \beta^2 = 1$.  When this additional condition is considered, we uncover further properties of its determinant and null space. \\

DNA matrices enrich the study of centrosymmetric structures by offering a novel polynomial-based formulation. This could lead to further investigations in linear algebra, combinatorial identities, and their potential applications in symbolic computation. 
Preliminary investigations  suggest that exploring an analogous context for other group actions and higher-dimensional spaces is a promising direction and may lead to new results on the structures and properties of centrosymmetric matrices. In fact, some analogous attempts to what we do in this article, applied to planar rotations in the Euclidean orthogonal group, yield matrices with a certain structure reminiscent of DNA matrices. However, these matrices do not fit into any well-known classes, such as symmetric, skew-symmetric, Hankel, Toeplitz, per-symmetric, anti-centrosymmetric matrices, and so on. Hence, as of now, questions in this direction remain open and unanswered. \\

The connection we establish here with invariant theory to discover a new class of centrosymmetric matrices suggests that attempts to develop results in the opposite direction are worthwhile. In fact, the problem of determining a Hilbert basis for the ring of invariant polynomials under the action of a Lie group is known in particular cases, for example when the Lie group is compact (\cite{hilbert1970theorie}, \cite{hilbert1970vollen}) and for  reductive groups (\cite{luna}). However, the existence of a Hilbert basis remains an open problem in many cases. \\

In the following sections, we develop the theoretical framework for DNA matrices, proving their centrosymmetry and exploring their structural properties. Section~\ref{sec: sample computations} presents explicit computations for DNA matrices of small orders, illustrating their construction and behavior. In Section~\ref{sec: main}, we establish their centrosymmetric structure rigorously and examine the role of polynomial invariants in their classification as singular or non-singular matrices. Section~\ref{sec: class} investigates additional algebraic properties when specific constraints are imposed on the parameters, such as the condition 
$\alpha^2 - \beta^2 = 1$, leading to further insights into their determinant and null space.

\section{Sample computations of centrosymmetric matrices} \label{sec: sample computations}

In this section we compute centrosymmetric matrices for six specific homogeneous polynomial equations: from linear equations to equations of degree 6.  \\

The ring of invariants is a graded algebra for which we expand degree by degree. We fix any $\theta \neq 0$ and consider  homogeneous polynomials $f_i : \mathbb{R}^{1+1} \to \mathbb{R}$ invariant under the (group generated by the) hyperbolic rotation $R_\theta$, $i = 1, \ldots 6$. Set  $\alpha = \cosh \theta$ and $\beta = \sinh \theta$. \\

Starting with degree 1, $f_1(x,y) = a_1 x + a_2 y$ and the invariance gives
\begin{eqnarray*}
	f_1(\alpha x + \beta y, \beta x + \alpha y) - f_1(x,y) = 0, \quad (x,y) \in \mathbb{R}^{1+1},
\end{eqnarray*}
or 
\begin{eqnarray*}
	\begin{cases}
		(\alpha - 1)a_1 + \beta a_2 = 0, \\
		\beta a_1 + (\alpha - 1)a_2 = 0.
	\end{cases}
\end{eqnarray*}
The coefficient matrix of this system is
\begin{eqnarray*}
	A_1  =  \begin{pmatrix} \alpha - 1 & \beta \\ \beta & \alpha - 1 \end{pmatrix},
\end{eqnarray*}
which is centrosymmetric for any pair of (independent) constants $\alpha$ and $\beta$. For  $\alpha = \cosh \theta$ and $\beta = \sinh \theta$,   we have that $\alpha \neq 0$, so it is straightforward that this is a nonsingular matrix. \\

We proceed similarly for the other cases,
\begin{eqnarray*}
	f_2(x,y) &=& a_1 x^2 + a_2 xy + a_3 y^2, \\
	f_3(x,y) &=& a_1 x^3 + a_2 x^2 y + a_3 xy^2 + a_4 y^3, \\
    	f_4(x,y) &=& a_1 x^4 + a_2 x^3 y + a_3 x^2 y^2 + a_4 x y^3 + a_5 y^4, \\
	f_5(x,y) &=& a_1 x^5 + a_2 x^4 y + a_3 x^3 y^2 + a_4 x^2 y^3 + a_5 xy^4 + a_6 y^5, \\
	f_6(x,y) &=& a_1 x^6 + a_2 x^5 y + a_3 x^4 y^2 + a_4 x^3 y^3 + a_5 x^2 y^4 + a_6 xy^5 + a_7 y^6.
\end{eqnarray*}
to obtain the coefficient matrices of the corresponding linear systems: 
\begin{eqnarray*}
A_2  =  	\begin{pmatrix}
		 \alpha^2-1 &  \alpha \beta & \beta^2 \\ 2\alpha \beta & \alpha^2+\beta^2-1 & 2\alpha \beta \\ \beta^2 & \alpha \beta & \alpha^2-1
	\end{pmatrix}
\end{eqnarray*}
\begin{eqnarray*}
	A_3 =  \begin{pmatrix}
		\alpha^3-1 & \alpha^2 \beta & \alpha \beta^2 & \beta^3 \\3 \alpha^2 \beta & \alpha^3+2 \alpha \beta^2-1 & 2 \alpha^2 \beta+ \beta^3 & 3 \alpha \beta^2 \\ 3 \alpha \beta^2 & 2 \alpha^2 \beta+\beta^3 &  \alpha^3+2 \alpha \beta^2-1 & 3 \alpha^2 \beta \\ \beta^3 & \alpha \beta^2 & \alpha^2 \beta & \alpha^3-1
	\end{pmatrix},
\end{eqnarray*}
\begin{eqnarray*}
	A_4 = \resizebox{0.9\linewidth}{!}{%
    $\begin{pmatrix}
		\alpha^4-1 & \alpha^3 \beta & \alpha^2 \beta^2 & \alpha \beta^3 &\beta^4 \\
		4 \alpha^3 \beta & \alpha^4+3 \alpha^2 \beta^2-1 & 2 \alpha^3 \beta+2 \alpha \beta^3& 3 \alpha^2 \beta^2+\beta^4 & 4\alpha \beta^3 \\ 
		6\alpha^2\beta^2 & 3\alpha^3 \beta+3 \alpha \beta^3 & \alpha^4+4 \alpha^2 \beta^2+\beta^4-1 & 3\alpha^3 \beta+3 \alpha \beta^3 & 6\alpha^2 \beta^2 \\ 4\alpha \beta^3 &3 \alpha^2 \beta^2+\beta^4 & 2\alpha^3 \beta+2 \alpha \beta^3 & \alpha^4+3 \alpha^2 \beta^2-1 & 4\alpha^3 \beta \\\beta^4 & \alpha \beta^3 & \alpha^2 \beta^2 & \alpha^3 \beta & \alpha^4-1
	\end{pmatrix},$}
\end{eqnarray*}
\begin{align*}
	& A_5 = \resizebox{0.9\linewidth}{!}{%
		$ \begin{pmatrix}
			\alpha^5-1 & \alpha^4 \beta & \alpha^3 \beta^2 & \alpha^2 \beta^3 & \alpha \beta^4 & \beta^5 \\ 5 \alpha^4 \beta & \alpha^5 +4 \alpha^3 \beta^2-1 & 2 \alpha^4 \beta+3 \alpha^2 \beta^3 & 3 \alpha^3 \beta^2+2 \alpha \beta^4 & 4 \alpha^2 \beta^3+\beta^5 & 5 \alpha \beta^4 \\ 10 \alpha^3 \beta^2 & 4 \alpha^4 \beta+6 \alpha^2 \beta^3 & \alpha^5+6 \alpha^3 \beta^2+3 \alpha \beta^4-1 & 3 \alpha^4 \beta+6 \alpha^2 \beta^3+\beta^5 & 6 \alpha^3 \beta^2+4 \alpha \beta^4 & 10 \alpha^2 \beta^3 \\ 10 \alpha^2 \beta^3 & 6 \alpha^3 \beta^2+4 \alpha \beta^4 & 3 \alpha^4 \beta+6 \alpha^2 \beta^3+\beta^5 & \alpha^5+6 \alpha^3 \beta^2+3 \alpha \beta^4-1 & 4 \alpha^4 \beta+6 \alpha^2 \beta^3 & 10 \alpha^3 \beta^2 \\ 5 \alpha \beta^4 & 4 \alpha^2 \beta^3+\beta^5 & 3 \alpha^3 \beta^2+2 \alpha \beta^4 & 2 \alpha^4 \beta+3 \alpha^2 \beta^3 & \alpha^5+4 \alpha^3 \beta^2-1 & 5 \alpha^4 \beta \\ \beta^5 & \alpha \beta^4 & \alpha^2 \beta^3 & \alpha^3 \beta^2 & \alpha^4 \beta & \alpha^5-1 
	\end{pmatrix},$
}
\end{align*}
\begin{align*}
	&A_6 = \resizebox{0.9\linewidth}{!}{%
		$  \begin{pmatrix}
			\alpha^6-1 & \alpha^5 \beta & \alpha^4 \beta^2 & \alpha^3 \beta^3 & \alpha^2 \beta^4 & \alpha \beta^5 & \beta^6 \\ 6 \alpha^5 \beta & \alpha^6+5 \alpha^4 \beta^2-1 & 2 \alpha^5 \beta+4 \alpha^3 \beta^3 & 3 \alpha^4 \beta^2+3 \alpha^2 \beta^4 & 4 \alpha^3 \beta^3+2 \alpha \beta^5 & 5 \alpha^2 \beta^4+\beta^6 & 6 \alpha \beta^5 \\ 15 \alpha^4 \beta^2 & 5 \alpha^5 \beta+10 \alpha^3 \beta^3 & \alpha^6+8 \alpha^4 \beta^2+6 \alpha^2 \beta^4-1 & 3 \alpha^5 \beta+9 \alpha^3 \beta^3+3 \alpha \beta^5 & 6 \alpha^4 \beta^2+8 \alpha^2 \beta^4+\beta^6 & 10 \alpha^3 \beta^3+5 \alpha \beta^5 & 15 \alpha^2 \beta^4 \\ 20 \alpha^3 \beta^3 & 10 \alpha^4 \beta^2+10 \alpha^2 \beta^4 & 4 \alpha^5 \beta+12 \alpha^3 \beta^3+4 \alpha \beta^5 & \alpha^6+9 \alpha^4 \beta^2+9 \alpha^2 \beta^4 +\beta^6-1 & 4 \alpha^5 \beta+12 \alpha^3 \beta^3+4 \alpha \beta^5 & 10 \alpha^4 \beta^2+10 \alpha^2 \beta^4 & 20 \alpha^3 \beta^3 \\ 15 \alpha^2 \beta^4 & 10 \alpha^3 \beta^3+5 \alpha \beta^5 & 6 \alpha^4 \beta^2+8 \alpha^2 \beta^4+\beta^6 & 3 \alpha^5 \beta+9 \alpha^3 \beta^3+3 \alpha \beta^5 & \alpha^6+8 \alpha^4 \beta^2+6 \alpha^2 \beta^4-1 & 5 \alpha^5 \beta+10 \alpha^3 \beta^3 & 15 \alpha^4 \beta^2 \\ 6 \alpha \beta^5 & 5 \alpha^2 \beta^4+\beta^6 & 4 \alpha^3 \beta^3+2 \alpha \beta^5 & 3 \alpha^4 \beta^2+3 \alpha^2 \beta^4 & 2 \alpha^5 \beta+4 \alpha^3 \beta^3 & \alpha^6+5 \alpha^4 \beta^2-1 & 6 \alpha^5 \beta \\ \beta^6 & \alpha \beta^5 & \alpha^2 \beta^4 & \alpha^3 \beta^3 & \alpha^4 \beta^2 & \alpha^5 \beta  & \alpha^6-1
		\end{pmatrix}$.
	}
\end{align*}

\quad

We emphasize that centrosymmetry is clearly a property in all these cases for  
arbitrary $\alpha$ and $\beta$. Now, by imposing $\alpha = \cosh \theta$ and $\beta = \sinh \theta$, from basic computational tools  we obtain
\[ \det A_{2k-1} \neq 0, \  \det A_{2k} = 0, \quad k = 1,2,3, \]
which is in agreement with the fact that  the invariant ring is generated by $x^2 - y^2$ and so nontrivial solutions of the linear system associated with $A_i$ exist precisely if $i$ is odd.

In general, for homogeneous $f_n(x,y) = \sum \limits_{i=0}^{n} a_i x^{n-i} y^i$, the invariance equation is
\begin{eqnarray} \label{eq: polynomial equation}  
	\sum \limits_{i=0}^{n} a_i \left( (\alpha x + \beta y)^{n-i} (\beta x + \alpha y)^i - x^{n-i} y^i \right) = 0,  
\end{eqnarray}  
whose associated coefficient matrix $A_n$  is a DNA matrix of order $n+1$, as defined in (\ref{dna}). In Theorem~\ref{thm: DNA is CSM} of the next section we prove that this is centrosymmetric.

\section{The centrosymmetric structure of DNA matrices} \label{sec: main}

In this section we prove that a matrix satisfying (\ref{dna}) is centrosymmetric.  \\

We start by recalling the Stifel's relation of binomial coefficients, 
\[
\binom{n-1}{k-1}+\binom{n-1}{k}=\binom{n}{k},
\]
for integers $n \geq k > 0$. In the next lemma we generalize this relation:

\begin{lemma}[$p$-Stifel's relation] \label{lemma: pStifel}
	Let $p$ be a non-negative integer. Then,
	\begin{eqnarray}
		\binom{n}{k}=\sum\limits_{i=0}^{p} \binom{p}{i} \binom{n-p}{k-(p-i)}. \label{stiffel}
	\end{eqnarray}
\end{lemma}

\begin{proof}
	The proof follows by induction. For $p=0$, we have
	\[
	\sum\limits_{i=0}^{0} \binom{p}{i} \binom{n-0}{k-(0-i)}=\binom{n}{k}.
	\]
	Suppose that (\ref{stiffel}) holds for $p$. We prove for $p+1$:
	\footnotesize
	\begin{eqnarray*}
		\sum\limits_{i=0}^{p+1} \binom{p+1}{i} \binom{n-(p+1)}{k-(p+1-i)} &=& \sum\limits_{i=0}^{p+1} \left(\binom{p}{i-1}+\binom{p}{i} \right) \binom{n-(p+1)}{k-(p+1-i)} \\
		&=& \sum\limits_{i=0}^{p+1} \binom{p}{i-1}  \binom{n-(p+1)}{k-(p+1-i)} \\
		&& + \sum\limits_{i=0}^{p+1} \binom{p}{i} \binom{n-(p+1)}{k-(p+1-i)}  \\
		&=& \sum\limits_{i=0}^{p}  \binom{p}{i} \left( \binom{n-(p+1)}{k-(p-i)}+ \binom{n-(p+1)}{k-(p+1-i)} \right).
	\end{eqnarray*}
	\normalsize
	Using Stifel's relation again,
	\[
	\binom{n-(p+1)}{k-(p-i)}+ \binom{n-(p+1)}{k-(p+1-i)} =\binom{n-p}{k-(p-i)},
	\]
	which completes the proof.
\end{proof}

\bigskip

Notice that we recover Stifel's relation if $p=1$:
	\begin{eqnarray*}
		\binom{n}{k}=\sum\limits_{i=0}^{1} \binom{1}{i} \binom{n-1}{k-(1-i)} = \binom{n-1}{k-1}+ \binom{n-1}{k}.
	\end{eqnarray*} 

To prove the main result, we also need:

\begin{lemma} \label{lemma: control of sums}
	The following equalities hold:
	\[
\begin{array}{l} 
		\sum\limits_{r=0}^{j-1} (-1)^{r+1} \binom{j-r}{i} \binom{j}{r}=0, \quad \text{}  j \neq i, \\ 
        \\
		\sum\limits_{r=0}^{j-1} (-1)^{r+1} \binom{j-r}{i} \binom{j}{r}=-1, \quad \text{}  j = i. 
	\end{array}
	\]
\end{lemma}

\begin{proof}
	For $j = i$, we have that $\binom{i-r}{i} = 0$ for all $r > 0$, since $i - r < i$ for $r > 0$. Therefore, we obtain
	\begin{eqnarray*}
		\sum\limits_{r=0}^{i-1} (-1)^{r+1} \binom{i-r}{i} \binom{i}{r} &=&  (-1)^{0+1} \binom{i-0}{i} \binom{i}{0}+\sum\limits_{r=1}^{i-1} (-1)^{r+1} \binom{i-r}{i} \binom{i}{r} \\
		&=& -1.
	\end{eqnarray*}
	Now, suppose that $j \neq i$. We consider two cases:
	
	\noindent {\bf Case $j < i$:}
	Since $j - r < i$ for all $r \geq 0$, it follows that $\binom{j-r}{i} = 0$ for all $r \geq 0$, which implies that
	\[
	\sum\limits_{r=0}^{j-1} (-1)^{r+1} \binom{j-r}{i} \binom{j}{r} = 0.
	\]
	{\bf Case $j > i$:}
	In this case, we rewrite the summation as:
	\begin{eqnarray*}
		\sum\limits_{r=0}^{j-1} (-1)^{r+1} \binom{j-r}{i} \binom{j}{r} = \sum\limits_{r=0}^{j-i} (-1)^{r+1} \binom{j-r}{i} \binom{j}{r},
	\end{eqnarray*}
	since $\binom{j-r}{i} = 0$ for $r > j - i$. Hence, we obtain
	\begin{eqnarray*}
		\sum\limits_{r=0}^{j-i} (-1)^{r+1} \binom{j-r}{i} \binom{j}{r} = \sum\limits_{r=0}^{j-i} (-1)^{r+1} \frac{j!}{i!r!(j-r-i)!}.
	\end{eqnarray*}
	Now, observe that if we define $a_r = (-1)^{r+1} \frac{1}{r!(2t+1-r)!}$, then $a_{2t+1-r} = -a_r$. If $j - i = 2t+1$ is odd, then
	\begin{eqnarray*}
		\sum\limits_{r=0}^{j-i} \frac{(-1)^{r+1} j!}{i!r!(j-r-i)!} &=& \frac{j!}{i!} \sum\limits_{r=0}^{2t+1} a_r \\
		&=& \frac{j!}{i!} \left( \sum\limits_{r=0}^{t} a_r+\sum\limits_{r=t+1}^{2t+1} a_r \right) \\
		&=& \frac{j!}{i!} \left( \sum\limits_{r=0}^{t} a_r+\sum\limits_{r=0}^{t} a_{2t+1-r} \right) = 0.
	\end{eqnarray*}
	Thus, for odd values of $j - i$,
	\[
	\sum\limits_{r=0}^{j-i} (-1)^{r+1} \frac{j!}{i!r!(j-r-i)!} = 0.
	\]
	Now, consider $b_r = (-1)^{r+1} \frac{1}{r!(2t-r)!}$. Then, $b_{2t-r} = b_r$. If $j - i = 2t$ is even, we obtain 
	\begin{eqnarray*}
		\sum\limits_{r=0}^{j-i} (-1)^{r+1} \frac{j!}{i!r!(j-r-i)!} &=& \frac{j!}{i!}  \sum\limits_{r=0}^{2t} b_r \\
		&=& \frac{j!}{i!} \left( \sum\limits_{r=0}^{t-1} b_r+(-1)^{t+1} \frac{1}{t!t!}+\sum\limits_{r=t+1}^{2t} b_r \right)  \\
		&=& \frac{j!}{i!} \left( \sum\limits_{r=0}^{t-1} b_r+(-1)^{t+1} \frac{1}{t!t!}+\sum\limits_{r=0}^{t-1} b_{2t-r} \right) \\
		&=& \frac{j!}{i!} \left( 2 \sum\limits_{r=0}^{t-1} b_r+(-1)^{t+1} \frac{1}{t!t!} \right).
	\end{eqnarray*}
	Thus, it suffices to prove that 
	\[
	2 \sum\limits_{r=0}^{t-1} b_r = -(-1)^{t+1} \frac{1}{t!t!}.
	\]
	For that,  define
	\[
	S_n = 2(b_0+b_1+ \cdots +b_n) = (-1)^{n+1} \frac{1}{n!t(2t-(n+1))!}.
	\]
	We have that
	\[
	S_0 = (-1) \frac{1}{0!t(2t-1)!} = - \frac{2}{2t(2t-1)!} = - \frac{2}{(2t)!} = 2b_0
	\]
and	$S_1 = \frac{1}{t(2t-2)!}$. On the other hand,
	\begin{eqnarray*}
		S_0+2b_1 &=& - \frac{2}{(2t)!}+\frac{2}{(2t-1)!} \\
		&=& - \frac{2}{(2t)!}+\frac{2(2t)}{(2t)(2t-1)!} \\
		&=& \frac{2(2t-1)}{2t(2t-1)!} = \frac{1}{t(2t-2)!} = S_1.
	\end{eqnarray*}
	Now suppose that the statement holds for $n = k$, i.e.,
	\[
	S_k = (-1)^{k+1} \frac{1}{k!t(2t-(k+1))!}.
	\]
	We prove it for $n = k+1$:
	\begin{eqnarray*}
		S_k+2b_{k+1} &=& (-1)^{k+1} \frac{1}{k!t(2t-(k+1))!}+(-1)^{k+2} \frac{2}{(k+1)!(2t-(k+1))!} \\
		&=& (-1)^{k+2} \left(\frac{2}{(k+1)!(2t-(k+1))!}-\frac{1}{k!t(2t-(k+1))!}\right) \\
		&=& (-1)^{k+2} \frac{2t-(k+1)}{(k+1)!t(2t-(k+1))!} \\
		&=& (-1)^{k+2} \frac{1}{(k+1)!t(2t-(k+2))!} = S_{k+1}.
	\end{eqnarray*}
	Therefore, 
	\[
	2\sum\limits_{r=0}^{t-1} b_r = S_{t-1} = -(-1)^{t+1} \frac{1}{t!t!}.
	\]
\end{proof}

\bigskip

At this stage, we can state the following nice property: 

\begin{theorem} 
	The sum of all elements in any column of a DNA matrix (\ref{dna}) of order $n+1$ is the polynomial $(\alpha + \beta)^n-1$.
\end{theorem}

\begin{proof}
From Lemma~\ref{lemma: control of sums}, it is a direct (although cumbersome ) calculation 	that the sum of the elements in the $\ell$-th column of the matrix is  given by
	\footnotesize
	\begin{eqnarray*}
		\sum \limits_{i=1}^{n+1} a_{i \ell}&=& \sum \limits_{i=1}^{n+1} \left( \sum \limits_{s=0}^{\ell-1} \binom{\ell-1}{s} \binom{n+1-\ell}{i-\ell+s} \alpha^{n-(i-\ell+2s)} \beta^{i-\ell+2s}+\sum \limits_{r=0}^{\ell-1} (-1)^{r+1} \binom{\ell-r}{i} \binom{\ell}{r} \right) \\
		&=& \sum \limits_{i=1}^{n+1} \sum \limits_{s=0}^{\ell-1} \binom{\ell-1}{s} \binom{n+1-\ell}{i-\ell+s} \alpha^{n-(i-\ell+2s)} \beta^{i-\ell+2s} - 1.
	\end{eqnarray*} \normalsize
Next, we determine the coefficient of $\alpha^{n-v} \beta^{v}$ for each $v=0,\dots,n$. Since we have the relation $i-\ell+2s = v$, it follows that $i = v + \ell - 2s$. We now consider different values of $s$:
	\begin{itemize}
		\item $s=0$: \footnotesize
		$$
		\sum_{0}^{0} \binom{\ell-1}{s} \binom{n-\ell+1}{v+\ell-(\ell-s)} \alpha^{n-(v+\ell-\ell+2s)} \beta^{v+\ell-\ell+2s} = \binom{\ell-1}{0} \binom{n-\ell+1}{v} \alpha^{n-v} \beta^v.
		$$
		\normalsize
		\item $s=1$:
		\begin{eqnarray*}
			& & \sum_{1}^{1} \binom{\ell-1}{s} \binom{n-\ell+1}{v+\ell-2-(\ell-s)} \alpha^{n-(v+\ell-2-\ell+2s)} \beta^{v+\ell-2-\ell+2s} \\
			&=& \binom{\ell-1}{1} \binom{n-\ell+1}{v-1} \alpha^{n-v} \beta^v.
		\end{eqnarray*}
		
		\item $s=2$:
		\begin{eqnarray*}
			& & \sum_{2}^{2} \binom{\ell-1}{s} \binom{n-\ell+1}{v+\ell-4-(\ell-s)} \alpha^{n-(v+\ell-4-\ell+2s)} \beta^{v+\ell-4-l+2s} \\
			&=& \binom{\ell-1}{2} \binom{n-\ell+1}{v-1} \alpha^{n-v} \beta^v.
		\end{eqnarray*}
		
		\item $s=\ell-2$:
		\begin{eqnarray*}
			& & \sum_{\ell-2}^{\ell-2} \binom{\ell-1}{s} \binom{n-\ell+1}{v-\ell+4-(\ell-s)} \alpha^{n-(v-\ell+4-\ell+2s)} \beta^{v-\ell+4-\ell+2s} \\
			&=& \binom{\ell-1}{\ell-2} \binom{n-\ell+1}{v-\ell+2} \alpha^{n-v} \beta^v.
		\end{eqnarray*}
		
		\item $s=\ell-1$:
		\begin{eqnarray*}
			& & \sum_{\ell-1}^{\ell-1} \binom{\ell-1}{s} \binom{n-\ell+1}{v-\ell+2-(\ell-s)} \alpha^{n-(v-\ell+2-\ell+2s)} \beta^{v-\ell+2-\ell+2s} \\
			&=& \binom{\ell-1}{\ell-1} \binom{n-\ell+1}{v-\ell+1} \alpha^{n-v} \beta^v.
		\end{eqnarray*}
	\end{itemize}
	Hence, the coefficient of $\alpha^{n-v} \beta^{v}$ is 
	\footnotesize
	\begin{eqnarray*}
		& &\left( \binom{\ell-1}{0} \binom{n-\ell+1}{v} +\binom{\ell-1}{1} \binom{n-\ell+1}{v-1}+ \cdots + \binom{\ell-1}{\ell-1} \binom{n-\ell+1}{v-\ell+1}  \right) \alpha^{n-v} \beta^{v} \\
		&=& \sum_{i=0}^{\ell-1} \binom{\ell-1}{i} \binom{n-\ell+1}{v-i} \alpha^{n-v} \beta^{v} \\
		&=& \sum_{i=0}^{\ell-1} \binom{\ell-1}{\ell-1-i} \binom{n-\ell+1}{v-(\ell-1-i)} \alpha^{n-v} \beta^{v} \\
		&=& \sum_{i=0}^{\ell-1} \binom{\ell-1}{i} \binom{n-\ell+1}{v-(\ell-1-i)} \alpha^{n-v} \beta^{v} \\
		&=& \binom{n}{v} \alpha^{n-v} \beta^{v},
	\end{eqnarray*}\normalsize
	where the last equality follows from the $p$-Stifel identity given by Lemma~\ref{lemma: pStifel}.
	We now replace $i = v + \ell - 2s$ into our expression and, noting that $0 \leq i-\ell+2s \leq n$, we finally obtain
	\begin{eqnarray*}
		& & \sum \limits_{i=1}^{n+1} \sum \limits_{s=0}^{\ell-1} \binom{\ell-1}{s} \binom{n+1-\ell}{i-\ell+s} \alpha^{n-(i-\ell+2s)} \beta^{i-\ell+2s}-1 \\
		&=& \sum \limits_{v=0}^{n} \sum \limits_{s=0}^{\ell-1} \binom{\ell-1}{s} \binom{n+1-\ell}{v+\ell-2s-\ell+s} \alpha^{n-(v+\ell-2s-\ell+2s)} \beta^{v+\ell-2s-\ell+2s}-1 \\
		&=& \sum \limits_{v=0}^{n} \sum \limits_{s=0}^{\ell-1} \binom{\ell-1}{s} \binom{n+1-\ell}{v-s} \alpha^{n-v} \beta^{v}-1 \\
		&=& \sum \limits_{v=0}^{n} \binom{n}{v}  \alpha^{n-v} \beta^{v}-1 = (\alpha + \beta)^n-1.
	\end{eqnarray*}
\end{proof}

\bigskip

We now present our main result:

\begin{theorem} \label{thm: DNA is CSM}
	A DNA matrix is a centrosymmetric matrix.
\end{theorem}

\begin{proof}
We prove that $a_{i,j} = a_{n+2-i, n+2-j}$ for $j \leq i$. By definition,
	\begin{eqnarray*}
		a_{n+2-i,n+2-j} &=&  \sum_{s=0}^{n+1-j} \binom{n+1-j}{s} \binom{j-1}{j-i+s} \alpha^{n-(j-i+2s)} \beta^{j-i+2s} \\ 
		&+& \sum_{r=0}^{n+1-j} (-1)^{r+1} \binom{n+2-j-r}{n+2-i} \binom{n+2-j}{r}.
	\end{eqnarray*}
	For the term 
	$$\binom{n+1-j}{s} \binom{j-1}{j-i+s} \neq 0,$$
	we require that $j-i+s \leq j-1$, which implies $s \leq i-1$. As a consequence, we get that $s = \min\{i-1, n+1-j\}$. Furthermore, from the condition $j-i+s \geq 0$, namely $s \geq i-j$, we also obtain  $s = \max\{0, i-j\}$.
	Now, assuming $j < i$, it follows that $s \geq i-j$, so we can rewrite
	\begin{eqnarray*}
		a_{n+2-i,n+2-j} &=&  \sum_{s=0}^{\min\{i-1,n+1-j\}} \binom{n+1-j}{i-j+s} \binom{j-1}{s} \alpha^{n-(j-i+2s)} \beta^{j-i+2s} \\ 
		&+& \sum_{r=0}^{n+1-j} (-1)^{r+1} \binom{n+2-j-r}{n+2-i} \binom{n+2-j}{r} \\
		&=& \sum_{s=0}^{\min\{i-1,n+1-j\}} \binom{n+1-j}{i-j+s} \binom{j-1}{s} \alpha^{n-(j-i+2s)} \beta^{j-i+2s},
	\end{eqnarray*}
	since, for $j < i$, we have $\sum \limits_{r=0}^{j-1} (-1)^{r+1} \binom{j-r}{i} \binom{j}{r} = 0$. 
	Additionally, for the term 
	$$\binom{n+1-j}{i-j+s} \binom{j-1}{s} \neq 0,$$ 
	we require $n+1-j \leq i-j+s$, which implies $s \leq n+1-i$. Moreover, since $s \leq j-1$, we deduce that $s \leq \min\{j-1, n+1-i\}$. 
	Comparing with $a_{i,j}$, we obtain:
	\begin{eqnarray*}
		a_{i,j} &=& \sum _{s=0}^{j-1} \binom{j-1}{s} \binom{n-j+1}{i-j+s} \alpha^{n-(i-j+2s)} \beta^{i-j+2s} \\
		&+& \sum _{r=0}^{j-1} \left( -1 \right) ^{r+1} \binom{j-r}{i} \binom{j}{r} \\
		&=& \sum _{s=0}^{\min\{j-1,n+1-i\}} \binom{j-1}{s} \binom{n-j+1}{i-j+s} \alpha^{n-(i-j+2s)} \beta^{i-j+2s},
	\end{eqnarray*}
	where we used the fact that $j < i$ and so  $j-r < i$ for all $r = 1, \dots, j-1$. Therefore, we conclude that $a_{i,j} = a_{n+2-i,n+2-j}$ for $j < i$. 
    
	Now, consider the case $i = j$. We have:
	\begin{eqnarray*}
		a_{n+2-i,n+2-i} &=& \sum_{s=0}^{n+1-i} \binom{n+1-i}{s} \binom{i-1}{s} \alpha^{n-2s} \beta^{2s} \\ 
		&+& \sum_{r=0}^{n+1-i} (-1)^{r+1} \binom{n+2-i-r}{n+2-i} \binom{n+2-i}{r}.
	\end{eqnarray*}
	For the term 
	$$\binom{n+1-i}{s} \binom{i-1}{s} \neq 0,$$ 
	we require that $s \leq n+1-i$ and $s \leq i-1$, which implies $s \leq \min\{i-1,n+1-i\}$. Furthermore, for 
	$$\binom{n+2-i-r}{n+2-i} \neq 0,$$
	it suffices that $n+2-i \leq n+2-i-r$, which forces $r = 0$. Hence, we obtain:
	\begin{eqnarray*}
		a_{n+2-i,n+2-i} = \sum_{s=0}^{n+1-i} \binom{n+1-i}{s} \binom{i-1}{s} \alpha^{n-2s} \beta^{2s} -1.
	\end{eqnarray*}
	On the other hand, for $a_{i,i}$, we have:
	\begin{eqnarray*}
		a_{i,i} &=& \sum_{s=0}^{n+1-i} \binom{i-1}{s} \binom{n+1-i}{s} \alpha^{n-2s} \beta^{2s} \\
		&+& \sum_{r=0}^{0} (-1)^{r+1} \binom{i-r}{i} \binom{i}{r} \\
		&=& \sum_{s=0}^{n+1-i} \binom{n+1-i}{s} \binom{i-1}{s} \alpha^{n-2s} \beta^{2s} -1 = a_{n+2-i,n+2-i},
	\end{eqnarray*}
which completes the proof.
\end{proof}

\section{The case $\alpha^2 - \beta^2 = 1$} \label{sec: class}

In this section we present properties of DNA matrices when $\alpha$ and $\beta$ in (\ref{dna}) satisfy  $\alpha = \cosh \theta$ and  $\beta = \sinh \theta$, for an arbitrary $\theta \neq 0.$ \\

 To deduce the results of the present section, we use explicitly the generator of the ring of invariant polynomials ${\mathbb R}^{1+1} \to {\mathbb R}$ under the invariance condition   (\ref{eq: invariance}) for a hyperbolic rotation $\theta \neq 0$. In the next subsection, we record the derivation of this generator.

\subsection{Hilbert basis for the ring of hyperbolic rotational invariant polynomials} \label{subseq: generator} 

For the Minkowski plane 
\[ {\mathbb R}^{1+1} = \{ (x,y) \ : \ ||(x,y)||_1 = (x^2 - y^2)^{1/2} \}, \]
we use hyperbolic complex coordinates: If we denote the complex hyperbolic numbers
\[\mathbb{C}_h = \{z= x + hy \mid x, y \in \mathbb{R}, h^2 = 1, h \notin \mathbb{R} \}, \]
then $\mathbb{R}^{1+1}\simeq \mathbb{C}_h$ with the standard action 
of a Lorenz rotation $\theta z=e^{h\theta} z$. 
 In ``real'' coordinates $(z, \bar{z})$, 
consider a polynomial
$$f(z)=\sum b_{i j} z^{i}\overline{z}^{j},$$   $b_{ij} \in \mathbb{C}_h$. Since this is a real-valued function,  then $\bar{b}_{ij} = b_{ji}.$ By $\theta$-invariance, it follows that
	$$b_{ij} e^{h\theta(i-j)}=b_{ij}.$$
But $\theta$ is a hyperbolic nonzero rotation, so  $a_{ij}=0$ or $i=j$. Hence,
$$f(x,y)= \sum_k b_{k} (x^2-y^2)^{k}.$$ 

From the above, we have just proved the following.

\begin{proposition} \label{prop:H basis}
The Hilbert basis for the ring of invariant polynomials ${\mathbb R}^{1+1} \to {\mathbb R}$ under the standard action of a hyperbolic rotation has one element, which is the norm 
\begin{equation} \label{eq: norm}
 ||(x,y)||_1^2 =  x^2 - y^2 .
\end{equation}
\end{proposition}

We remark that this is a simple result related to the complex problem of existence of finite basis of invariant rings under the action of Lie groups. Hilbert initially proved that when the Lie group is compact, the ring of invariant polynomials admits a finite basis (\cite{hilbert1970theorie}, \cite{hilbert1970vollen}). Later, Luna extended the result to reductive groups (\cite{luna}). For counterexamples, we refer to Nagata’s counterexamples (\cite{nagata195914}).

\subsection{Characterization of singular matrices} \label{subseq: properties}

We start by establishing that, inside the present class, the classification as singular or non-singular is determined precisely by the parity of their order:

\begin{proposition} \label{detpar_en}
	For $n \geq 1$, let $A_n$ be a DNA matrix of order $n+1$ given in (\ref{dna}) with $\alpha^2 - \beta^2 = 1$. Then:
	\begin{itemize}
		\item If $n$ is odd, $\det A_n \neq 0$.
		\item If $n$ is even, $\det A_n = 0$.
	\end{itemize}
\end{proposition}

\begin{proof}
For each $n$,  $A_n$ is the coefficient matrix of the homogeneous linear system generated by the polynomial equation (\ref{eq: polynomial equation}) in the $a_i$'s variables, $i = 0 \ldots, n$. Now, it follows directly from Proposition~\ref{prop:H basis} that   if $n$ is odd, then  $f_n \equiv 0$, that is, $\det A_n \neq 0$. Also, for any $n$ even,  there are non-zero homogeneous polynomials of degree $n$, which implies that   $\det A_n = 0$.  
\end{proof}

The singularity of odd-order matrices suggests the study of their  nontrivial null space. In this case,  we obtain the solution explicitly: 
\begin{theorem} \label{solucaosistema_en} For	$n$ even, let $A_n$ be a DNA matrix of order $n+1$ with entries given in (\ref{dna}) with $\alpha^2 - \beta^2 = 1$. Then, the null space of $A_n$ is generated by the vector $(x_1, \ldots, x_{n+1}) $, where
	\begin{eqnarray*}
		\begin{cases} 
			x_{2k}=0, \quad \text{for } k=1, \dots, \frac{n}{2}  \, \\
			x_{2k-1}=(-1)^{k-1} \binom{\frac{n}{2}}{k-1}, \quad \text{for } k=1, \dots, \frac{n}{2}+1.
		\end{cases}
	\end{eqnarray*}
\end{theorem}
\begin{proof}
This is direct from the fact that any invariant homogeneous polynomial of even degree $n$  is of the form \( f(x,y) = a (x^2 - y^2)^{n/2} \), for some $a \in {\mathbb R}$, by (\ref{eq: norm}).
\end{proof}

\noindent Theorem~\ref{solucaosistema_en} clearly reveals the combinatorial symmetry inherent in the algebraic expression of the null space of a DNA matrix.\\

\subsection{From theory to a special case} \label{sec: illustration} 

In this subsection we illustrate the elements discussed in the previous subsection for $\theta = \ln 2$, so  $$\alpha = 5/4, \quad \beta = 3/4.$$ The data in Table~1 have been computed with Maple (Maplesoft, 2024) up to degree 10.
\begin{table}[H] \label{table}
	\centering
	\caption{The case  $\theta = \ln 2$.}
	\resizebox{\textwidth}{!}{ 
		\begin{tabular}{c|c|c|c|c}
			\hline \hline
			$\begin{array}{c} \text{Polynomial} \\ \text{degree} \end{array}$ & Matrix & Determinant & Null vector & $\begin{array}{c} \text{Invariant} \\ \text{polynomial} \end{array}$ \\ \hline \hline
			
			1 & $A_1$ & $-\frac{1}{2}$ &  - & - \\ \hline
			
			2 & $A_2$ & $0$ & (1, 0, -1) & $a (x^2-y^2)$ \\ \hline
			
			3 & $A_3$ & $\frac{49}{16}$ & - & - \\ \hline
			
			4 & $A_4$ & $0$ & $\begin{array}{c} (1, 0, -2, 0 , 1) \\ \end{array}$ & $a (x^2-y^2)^2$ \\ \hline 
			
			5 & $A_5$ & $-\frac{47089}{512}$ & - & - \\ \hline
			
			6 & $A_6$ & $0$ & $\begin{array}{c} (1, 0, -3, 0, 3, 0, -1) \\ \end{array}$ & $a (x^2-y^2)^3$ \\ \hline
			
			7 & $A_7$ & $\frac{759498481}{65536}$ & -  & - \\ \hline
			
			8 & $A_8$ & $0$ & $\begin{array}{c} (1, -4, 0, 6, 0, -4, 1) \\ \end{array}$ & $a (x^2-y^2)^4$ \\ \hline
			
			9 & $A_9$ & $\frac{198321002857201}{33554432}$ & - & - \\ \hline
			
			10 & $A_{10}$ & $0$ & $\begin{array}{c} (1, -5, 0, 10, 0, -10, 0, 5, -1)  \end{array}$ & $a (x^2-y^2)^5$ \\
			
			\hline \hline
		\end{tabular}
	}
\end{table}

Finally, we emphasize that the correlation between the centrosymmetry of DNA matrices considered here and invariant polynomials in two variables holds for the standard action on the plane of any matrix of the form
\[ \begin{pmatrix}
\alpha & \beta \\ \beta & \alpha \end{pmatrix}, \]
for any $\alpha, \beta \in {\mathbb R}$. The additional properties of DNA matrices presented in this section are obtained specifically under the restriction to polynomials on the Minkowski plane that are invariant under a Lorenz rotation.

\section*{Acknowledgements}

The work of L.N. was partially supported by FAPESP, grant number 2022/12906-3 and the work of M.M. was partially supported by FAPESP, grant number 2019/21181-0.

\bibliographystyle{abbrv}

\bibliography{references}

\end{document}